\documentclass[12pt]{article}
\usepackage[utf8]{inputenc}
\usepackage{enumerate}

\usepackage{upgreek}

\usepackage{authblk}
\usepackage{amsmath,amssymb,amsthm}
\usepackage[margin=2.5cm]{geometry}
\usepackage{mathptmx}

\newtheorem{thm}{Theorem}
\newtheorem{dfn}[thm]{Definition}

\newcommand{\R}{\ensuremath{\mathbb{R}}}
\newcommand{\supp}{\mathrm{supp}}
\newcommand{\FEAPS}{F\hspace{-0.6mm}E_\mathrm{APS}}
\newcommand{\APS}{\mathrm{APS}}
\newcommand{\aAPS}{\mathrm{aAPS}}
\newcommand{\Adach}{\widehat{\mathrm{A}}}
\newcommand{\dV}{\, dV}
\newcommand{\Pspin}{P_\mathrm{Spin}}

\usepackage{etex}

\newtheorem{theorem}[thm]{Theorem}

\newtheorem{lemma}[thm]{Lemma}
\newtheorem{definition}[thm]{Definition}

\usepackage{hyperref}

\newcommand{\arxivref}[1]{\href{http://www.arxiv.org/abs/#1}{{arXiv.org:#1}}}

\newcommand{\blue}[1]{{\color{blue}#1}}

\usepackage{mathrsfs}
\usepackage{eucal}

\newcommand{\NatNum}{\mathbb{N}}

\newcommand{\Naturals}{\mathbb{N}}

\newcommand{\Reals}{\mathbb R}
\newcommand{\Complex}{\mathbb C}

\newcommand{\Id}{\text{Id}}

\usepackage{ipa}

\usepackage{adjustbox}
\usepackage{tikz}
\usepackage{rotating} 
\usetikzlibrary{cd}
\usetikzlibrary{arrows} 
\usetikzlibrary{patterns}
\usetikzlibrary{shapes,snakes}
\usetikzlibrary{shapes}
\usetikzlibrary{decorations.pathmorphing}

\usetikzlibrary{datavisualization} 
\usetikzlibrary{datavisualization.formats.functions}
\usepackage{pgfmath}

\newcounter{mnotecount}[section]

\newcounter{mymnotecount}[section]
\renewcommand{\themymnotecount}{\thesection.\arabic{mymnotecount}}
\newcommand{\mymnote}[1]{\protect{\stepcounter{mnotecount}}${\raisebox{0.5\baselineskip}[0pt]{\makebox[0pt][c]{\color{blue}{\tiny\em$\bullet$\themymnotecount}}}}$\marginpar{\raggedright\tiny\em$\!\bullet$\themymnotecount:
\blue{#1}}\ignorespaces}

\renewcommand{\mymnote}[1]{}

\newcommand{\Scri}{\mathcal{I}}

\newcommand{\half}{\tfrac{1}{2}}         

\newcommand{\tr}{\text{tr}}

\newcommand{\Symop}{\mathbb{S}}
\newcommand{\FF}{\mathcal F}

\newcommand{\NPl}{l}
\newcommand{\NPn}{n}
\newcommand{\NPm}{m}
\newcommand{\NPmbar}{\bar m} 
\newcommand{\KSigma}{\Sigma} 
\newcommand{\KDelta}{\Delta}

\newcommand{\sDiv}{\mathscr{D}}
\newcommand{\sCurl}{\mathscr{C}}
\newcommand{\sCurlDagger}{\mathscr{C}^\dagger}
\newcommand{\sTwist}{\mathscr{T}}

\newcommand{\sfrak}{\mathfrak{s}}
\newcommand{\GenVec}{\nu} 

\newcommand{\mass}{m}

\newcommand{\Indicator}{\mathbf{1}}

\newcommand{\solu}{\psi}
\newcommand{\localiseAwayFromPhotonOrbits}{\Indicator_{r\not\eqsim 3M}}

\newcommand{\ModelEnergyThree}{E_{\text{model},3}}

\newcommand{\newepsilon}[3]{\newcommand{#1}{#2}}
 
\newepsilon{\ConsantIntroUniformBound}{C}{C_1}
\newepsilon{\epsilonSlowRotationIntroUniformBound}{\bar{a}}{\bar{a}_1}
\newepsilon{\ConstantIntroDecay}{C}{C_2}
\newepsilon{\epsilonSlowRotationIntroDecay}{\bar{a}}{\bar{a}_2}
\newepsilon{\ConsantIntroUniformBoundsolu}{C}{C_3}
\newepsilon{\epsilonSlowRotationIntroUniformBoundSolu}{\bar{a}}{\bar{a}_3}
\newepsilon{\ConstantSTwoSobolev}{C}{C_{S}}
\newepsilon{\ConstantLEpsilonL}{C}{C_{\OpL}}
\newepsilon{\epsilonInvBlendLocation}{r_{\fnBlend}}{r_{\fnBlend}}
\newepsilon{\epsilonSlowRotationNearEnergyPositive}{\bar{a}}{\bar{a}_{T}}
\newepsilon{\epsilonMorawetzSlowRotation}{\bar{a}}{\bar{a}_{\text{Morawetz}}}
\newepsilon{\epsilonMorawetzPhotonWidth}{\bar{r}}{\bar{r}_{\text{Morawetz}}}
\newepsilon{\epsilonMorawetzTurnOndtSquared}{\epsilon_{\dt^2}}{\epsilon_{\dt^2}}
\newepsilon{\epsilonMorawetzTurnOndtSquaredUpperBound}{\overline{\epsilonMorawetzTurnOndtSquared}}{\epsilon_{\dt^2,\text{upper bound}}}
\newepsilon{\epsilonInverseErrorInMorawetzBeforeHardy}{C}{C_\text{Morawetz error}}
\newepsilon{\epsilonInvIntegratedMorawetzHomo}{C_1}{C_{\text{Homogeneous Morawetz}}}
\newepsilon{\epsilonInvIntegratedMorawetzHomoB}{C_2}{C_{\text{Homogeneous Morawetz,2}}}
\newepsilon{\epsilonHardySlowRotation}{\bar{a}}{\bar{a}_\text{Hardy}}

\newepsilon{\epsilonInvIntegratedMorawetz}{C}{C_{\text{Morawetz}}}
\newepsilon{\epsilonBoundedEnergySlowRotation}{\bar{a}}{\bar{a}_E}
\newepsilon{\epsilonInvBoundedEnergy}{C}{C_E}
\newepsilon{\epsilonInvBoundedEnergyInner}{C'}{C_E'}
\newepsilon{\epsilonSlowRotationDominantEnergy}{\bar{a}}{\bar{a}_{DEC}}
\newepsilon{\ConstantDominantEnergy}{C}{C_{DEC}}
\newepsilon{\epsilonSlowRotationRelationBetweenEnergies}{\bar{a}}{\bar{a}_{\tilde{E}}}
\newepsilon{\ConstantRelationBetweenEnergies}{C}{C_{\tilde{E}}}
\newepsilon{\epsilonSlowRotationKTimelike}{\bar{a}}{\bar{a}_{TL}}
\newepsilon{\epsilonSlowRotationKEnergy}{\bar{a}}{\bar{a}_{KE}}

\newepsilon{\epsilonSlowRotationLightConeLocalise}{\bar{a}}{\bar{a}_{LC}}
\newepsilon{\ConstantLightConeLocalisation}{C}{C_{LC}}

\newepsilon{\epsilonSlowRotationInK}{\bar{a}}{\bar{a}_K}
\newepsilon{\ConstantInK}{C}{C_K}
\newepsilon{\ConstantInKexp}{C'}{C_K'}
\newepsilon{\epsilonSlowRotationOtherSfcs}{\bar{a}}{\bar{a}_{OS}}
\newepsilon{\ConstantOtherSfcs}{C}{C_{OS}}
\newepsilon{\ConstantOtherSfcsexp}{C'}{C_{OS}'}

\newepsilon{\epsilonSlowRotationStationaryDecay}{\bar{a}}{\bar{a}_{\text{SD}}}
\newepsilon{\ConstantStationaryDecayexp}{C'}{C_{SD}'}

\newepsilon{\epsilonSlowRotationND}{\bar{a}}{\bar{a}_{ND}}
\newepsilon{\ConstantND}{C}{C_{\text{ND}}}
\newepsilon{\ConstantNDexp}{C'}{C_{\text{ND}}'}
\newepsilon{\epsilonSlowRotationFarAlmostNullEnergy}{\bar{a}}{\bar{a}_{FANE}}
\newepsilon{\ConstantFarAlmostNullEnergy}{C}{C_{FANE}}
\newepsilon{\ConstantFarAlmostNullEnergyexp}{C'}{C_{FANE}'}
\newepsilon{\epsilonSlowRotationFarDecay}{\bar{a}}{\bar{a}_{FD}}
\newepsilon{\ConstantFarDecay}{C}{C_{FD}}
\newepsilon{\ConstantFarDecayExp}{C'}{C_{FD}'}

\newcommand{\normPtwiseTn}[2]{|#2|_{#1}}

\newcommand{\OpL}{\mathcal{L}}

\newcommand{\rp}{r_+}

\newcommand{\dr}{\partial_r}

\newcommand{\dt}{\partial_t}
\newcommand{\dtheta}{\partial_\theta}
\newcommand{\dphi}{\partial_\phi}

\newcommand{\dAng}{{\nabla\!\!\!\!/}}

\newcommand{\di}{\mathrm{d}} 
\newcommand{\diFourNatural}{\gVol\di^4 x}

\newcommand{\diThree}{\di^3\mu}
\newcommand{\diFour}{\di^4\mu}

\newcommand{\gVol}{\sqrt{|\gMetric|}}

\newcommand{\SymGeneraln}[1]{\mathbb{S}_{#1}}
\newcommand{\CQA}{S}
\newcommand{\ua}{{\underline{a}}}
\newcommand{\ub}{{\underline{b}}}

\newcommand{\Lie}{\mathcal{L}}

\newcommand{\SOSym}{\SymGeneraln{2}}
\newcommand{\OpQ}{Q}

\newcommand{\hst}[1]{\Sigma_{#1}}

\newcommand{\GeodesicEnergy}{\boldsymbol{e}}
\newcommand{\GeodesicLz}{\boldsymbol{\ell_z}}

\newcommand{\KCarter}{\boldsymbol{k}}

\newcommand{\nq}{{n_\OpQ}}
\newcommand{\nt}{{n_t}}
\newcommand{\np}{{n_\phi}}

\newcommand{\EnergyCrudeFArg}[1]{E_{}[#1]}
\newcommand{\EnergyCrudeFIArg}[1]{E_{\text{FI}}[#1]}
\newcommand{\BulkMorawetzF}{B_{\pm}}
\newcommand{\BulkFILowerOrder}{B_{0}}
\newcommand{\BulkFIOne}{B_1}
\newcommand{\BulkFITwo}{B_{2,0}}

\newcommand{\inthst}[1]{\int_{\hst{#1}}}
\newcommand{\intBulk}[1]{\int_0^{#1}\inthst{t}}

\newcommand{\diomega}{\sin\theta\di\theta\di\phi}

\newcommand{\diFourFI}{\di^4\mu_{\mathrm{FI}}}

\newcommand{\pt}{\partial_t}
\newcommand{\pr}{\partial_r}

\newcommand{\pAng}{\not\!\nabla}

\newcommand{\FinalTime}{T}

\newcommand{\MaxF}{\textrm{F}}

\newcommand{\MaxFTotal}{\textrm{F}_{\text{total}}}

\newcommand{\phii}{\phi_i}

\newcommand{\gMetric}{g}

\newcommand{\Stwo}{\mathbb{S}^2}

\newcommand{\solub}{\bar{\Upsilon}}

\newcommand{\soluTotal}{\Upsilon_{\text{total}}}

\newcommand{\epsilonMain}{\frac{|a|}{M}}
\newcommand{\epsilonSlowRotationIntro}{\bar{\epsilon}_a}
\newcommand{\ConstantIntro}{C}

\newcommand{\phihat}{\widehat{\phi}}

\newcommand{\MidMet}{\mathcal{M}}

\begin{document}


  \author[1]{Lars Andersson}
  \author[2]{Christian B{\"a}r} 
  \affil[1]{Albert Einstein Institut, Am M{\"u}hlenberg 1, 14476 Potsdam, Germany}
  \affil[2]{Universit{\"a}t Potsdam, Institut f{\"u}r Mathematik, Karl-Liebknecht-Str. 24-25, 14476 Potsdam, Germany}
  \title{Wave and Dirac equations on manifolds}

\maketitle

  \abstract{We review some recent results on geometric equations on Lorentzian manifolds such as the wave and Dirac equations.
  This includes well-posedness and stability for various initial value problems, as well as results on the structure of these equations on black-hole spacetimes (in particular, on the Kerr solution), the index theorem for hyperbolic Dirac operators and properties of the class of Green-hyperbolic operators.
  
  
  }


\section*{Introduction}

Spacetime  
in general relativity is modeled by a Lorentzian manifold.
The physically relevant field equations are geometric partial differential equations on these manifolds.
The most prominent examples are wave and Dirac equations.
The present article surveys some progress in the study of these geometric equations obtained in the past years within the research network SFB 647.

Wave and Dirac equations are hyperbolic partial differential equations (PDEs) and the other equations under consideration share some aspects of hyperbolicity as well.
In order to obtain a reasonable analytic solution theory we have to exclude causal pathologies like closed timelike loops in the underlying manifold.
The right class of Lorentzian manifolds turns out to be that of globally hyperbolic ones.
In the first section we quickly review their most important properties.

Considering equations of hyperbolic type on globally hyperbolic spacetimes leads to the study of initial value problems. 
The second section provides a discussion of the Cauchy problem for wave equations and for Dirac equations, the Goursat problem (i.e.\ the characteristic initial value problem) for wave equations and the Cauchy problem for the vacuum Einstein equation.
From a physical point of view, well-posedness of these initial value problems is essential for predictability.

The singularity theorems of Hawking and Penrose indicate that vacuum spacetimes with strong gravitational fields will typically be singular in the sense that they have a non-trivial Cauchy horizon. 
According to the weak cosmic censorship conjecture such singularities in isolated systems must be hidden behind event horizons.
Such models are known as black hole solutions.
The most prominent examples are the Schwarzschild solution describing a static black hole and the Kerr solutions describing rotating black holes (in the subextreme case).
The form of the spacetime metric of these models is known explicitly.

The Kerr black hole solution is conjectured to be the unique stationary, asymptotically flat, vacuum spacetime containing a non-degenerate black hole. 
In addition to being unique, the Kerr spacetime is conjectured to be dynamically stable, in the sense that the maximal Cauchy development of any sufficiently small perturbation of subextreme Kerr Cauchy data is asymptotic to the future, to a subextreme member of the Kerr family. 

We may refer to the above conjectures as the \emph{black hole uniqueness}, and \emph{black hole stability} conjectures. Both of these conjectures, which are of central importance in general relativity and astrophysics are open in full generality, in spite of much work on the topic during the last 50 years. 

The importance of the Kerr solution, and an analysis of its properties, is explained by the fact that according to our current understanding, black holes are ubiquitous in the universe, in particular most galaxies have a supermassive black hole at their center, and these play an important role in the life of the galaxy. Also our galaxy has at its center a very compact object, Sagittarius A*, with a diameter of less than one astronomical unit, and a mass estimated to be $10^6$ $\mass_{\odot}$. Evidence for this includes observations of the orbits of stars in its vicinity.

In section 3 we provide background on massless field equations on black hole spacetimes and on the geometry of the Kerr solution.
The subsequent section contains new results on integrated local energy decay for solutions of massless field equations on Kerr spacetimes.
Moreover, higher order conservation laws for Maxwell equations are derived and integrability conditions on the linearized Einstein equations on Kerr are established.
All of this relies on the existence of Killing-Yano forms on Kerr spacetimes.

In section 5 we describe a recently found analogue to the Atiyah-Patodi-Singer index theorem on globally hyperbolic Lorentzian spacetimes.
In contrast to the original Riemann version of this index theorem, the spectral boundary conditions admit a natural physical interpretation in the Lorentzian case in terms of a particle-antiparticles splitting.
We indicate how the Lorentzian index theorem can be applied to compute the chiral anomaly in quantum field theory.
The chiral anomaly is concerned with the fact that particles may be created by gravity or an external field without also creating their corresponding antiparticles.

Well-posedness of the Cauchy problem implies the existence of Green's operators which are solution operators satisfying certain causality conditions.
In the last section we turn things around and deal with those differential operators which possess Green's operators even if they do not have a well-posed Cauchy problem.
This class is very large; it is closed under many natural operations producing new operators out of given ones.
Yet much of the analytic theory is shown to carry over to these equations.
This is important for quantizing these fields in the framework of algebraic quantum field theory on curved backgrounds.


\section{Globally-hyperbolic Lorentzian manifolds}

We describe the class of Lorentzian manifolds on which wave equations have reasonable analytic properties.
For a general introduction to Lorentzian geometry see one of the standard textbooks such as \cite{MR1384756,MR0424186,MR719023}.

Throughout this section let $M$ denote a connected time-oriented Lorentzian manifold.
For $A\subset M$ let $J^+(A)$ be the \emph{causal future} of $A$, i.e.\ the set of all points of $M$ which can be reached by future-directed causal curves starting in $A$.
Similarly, one defines the \emph{causal past} $J^-(A)$.
We also write $J(A) = J^+(A) \cup J^-(A)$.
A subset $A\subset M$ is called \emph{past compact} if $A\cap J^-(p)$ is compact for all $p\in M$.
Similarly, one defines \emph{future compact} subsets.

A subset $S\subset M$ is called a \emph{Cauchy hypersurface} if every inextendible timelike curve meets $S$ exactly once.
There is no a-priori regularity assumption on Cauchy hypersurfaces but it turns out that they are Lipschitz hypersurfaces, i.e.\ can locally be written as graphs of a Lipschitz function.
Moreover, any Cauchy hypersurface is a closed subset of $M$, every inextendible causal curve meets $S$ (possibly more often than once) and any two Cauchy hypersurfaces of $M$ are homeomorphic.

We say that $M$ satisfies the \emph{causality condition} if $M$ does not contain a closed causal loop.

\begin{thm}[Bernal, S\'anchez]\label{thm:globhyp}
The following are equivalent:
\begin{enumerate}[(1)]
\item
There exists a Cauchy hypersurface in $M$;
\item
$M$ satisfies the causality condition and the intersection $J^+(p)\cap J^-(p)$ is compact for all $p\in M$;
\item
The manifold $M$ is isometric to $\R\times S$ with metric $-N^2dt+g_t$ where $N$ is a smooth positive function (the \emph{lapse function}), $g_t$ is a Riemannian metric on $S$ depending smoothly on the parameter $t$ and all sets $\{t_0\}\times S$ are Cauchy hypersurfaces in $M$.
\end{enumerate}
\end{thm}

The implication (3)$\Rightarrow$(1) is trivial and the implication (1)$\Rightarrow$(2) has been known for a long time, see e.g.\ \cite[Cor.~39, p.~422]{MR719023}.
The implication (2)$\Rightarrow$(3) is a remarkable structural result due to A.\ Bernal and M.\ S\'anchez (combine \cite[Thm.~1.1]{MR2163568} with \cite[Thm.~3.2]{MR2294243}).

Manifolds satisfying the conditions in Theorem~\ref{thm:globhyp} are called \emph{globally hyperbolic}.
Many important relativistic models are globally hyperbolic such as Minkowski spacetime, Robertson-Walker spacetimes (in particular, Friedman cosmological models), the interior and the exterior Schwarschild models, deSitter spacetime and many more.


\section{Initial value problems}

We now describe some analytical facts on initial value problems for linear wave equations.

\subsection{Cauchy problem for linear wave equations}

Throughout this section let $M$ be a globally hyperbolic Lorentzian manifold and let $E\to M$ be a vector bundle.
A linear differential operator $P$ of second order with smooth coefficients acting on sections of $E$ is called \emph{normally hyperbolic} if its principal symbol is given by the Lorentzian metric, i.e.\ in local coordinates $P$ takes the form
$$
P = 
-\sum_{i,j=1}^n g^{ij}(x)\frac{\partial^2}{\partial x^i \partial x^j} 
+ \sum_{j=1}^n A_j(x) \frac{\partial}{\partial x^j} + B(x) \, .
$$
A linear wave equation is an equation of the form $Pu=f$ with given $f$ and an unknown section $u$.
By the \emph{Cauchy problem} we mean the problem of solving such a wave equation while imposing initial value conditions of zeroth and first order.
More precisely, let $S\subset M$ be a smooth spacelike Cauchy hypersurface and let $\mathfrak{n}$ be the future-directed unit normal vector field along $S$.
We fix a connection $\nabla$ on $E$.
Now the Cauchy problem consists of finding a section $u$ such that 
\begin{equation}
 \left\{\begin{array}{cccl}
Pu&=&f&\textrm{ on }M,\\
u&=&u_0&\textrm{ along }S,\\
\nabla_\mathfrak{n} u&=&u_1&\textrm{ along }S,
\end{array}\right.
\label{eq:CauchyProblem}
\end{equation}
for given $f$, $u_0$, and $u_1$.

\begin{thm}\label{thm:cauchyglobhyp}
Under the assumptions of this section the following holds:
\begin{enumerate}[(i)]
\item \emph{(Existence and uniqueness)}
For each  $u_0, u_1 \in C^\infty_c(S,E)$ and for each
$f\in C^\infty_c(M,E)$ there exists a unique $u\in C^\infty(M,E)$ solving \eqref{eq:CauchyProblem}.
\item \emph{(Stability)}
The solution $u$ depends continuously on the data $f$, $u_0$, and $u_1$.
\item \emph{(Finite propagation speed)}
The solution $u$ satisfies $\supp(u) \subset J(K)$ where $K=\supp(u_0)\cup \supp(u_1) \cup \supp(f)$.
\end{enumerate}
\end{thm}

A proof can be found in \cite[Thms.~3.2.11 and 3.2.12]{MR2298021}.
Assertions (i) and (ii) are often combined by stating that the Cauchy problem is \emph{well posed}.

This theorem also holds for data of lower regularity; we can, for instance, impose Sobolev regularity.
More precisely, the theorem holds for $u_0\in H^k_c(S,E)$, $u_1\in H^{k-1}_c(S,E)$, and $f$ being locally $L^2$ in time, $H^{k-1}$ in space and having spatially compact support, $k\in\R$.
The solution $u$ will then be in a suitable space of finite-energy sections, see \cite[Thm.~8]{MR3316713} for details.

\subsection{Goursat problem for linear wave equations}

There is an interesting borderline case between initial value problems and boundary value problems, namely the \emph{characteristic initial value problem} also known as the \emph{Goursat problem}.
Here one replaces the smooth spacelike Cauchy hypersurface along which the initial conditions have been required in the Cauchy problem by a partial Cauchy hypersurface\footnote{A \emph{partial Cauchy hypersurface} is a closed Lipschitz hypersurface which is met by each timelike curve at most once.} $S$ which is characteristic meaning that the induced metric on $S$ is degenerate.
Typical examples are of the form $S=\partial J^+(A)$ for any nonempty subset $A\subset M$.
It turns out that one can prescribe the initial values along $S$ but one cannot impose any condition on the derivative.

\begin{thm}[{B\"ar, Tagne Wafo \cite[Thm.~23]{MR3316713}}]\label{thm:CharacteristicUniqueness}
Let $S\subset M$ be a characteristic partial Cauchy hypersurface.
Assume that $J^+(S)$ is past compact.

Then for any locally $L^2$-section $f$ on $M$ with spatially compact support and any $u_0\in H^1_c(S,E)$ there exists a section $u$ on $M$ in  a suitable finite-energy space such that $Pu=f$ on $J^+(S)$ and $u|_S=u_0$.
On $J^+(S)$, $u$ is unique.
\end{thm}

Examples show easily that the assumption on $J^+(S)$ being past compact cannot be dropped.

\subsection{Cauchy problem for Dirac equations}

Theorem~\ref{thm:cauchyglobhyp} allows us to derive well-posedness also for the Cauchy problem for Dirac equations.
A linear differential operator $D$ of first order acting on sections of $E$ is called a \emph{Dirac-type operator} if $P=D^2$ is normally hyperbolic.
Examples are the classical Dirac operators acting on sections of the spinor bundle $E=\Sigma M$ (see \cite{MR701244} for details) or, more generally, on sections of a twisted spinor bundle $E=\Sigma M\otimes F$ where $F$ is any ``coefficient bundle'' equipped with a connection.

Particular examples are the \emph{Euler operator} $P=i(d-\delta)$ on $E=\bigoplus_k \Lambda^kT^*M$ and, in dimension $\dim(M)=4$, the \emph{Buchdahl operators} on $\Sigma M\otimes \Sigma_+^{\odot k}M$, see \cite[Sec.~2.5]{MR3289848} for details.
The following is the basic well-posedness property for the Cauchy problem for Dirac operators.

\begin{thm}\label{thm:cauchyglobhypDirac}
Let $D$ be a Dirac-type operator acting on sections of a vector bundle $E$ over a globally hyperbolic manifold $M$. 
Let $S\subset M$ be a smooth spacelike Cauchy hypersurface.
Then
\begin{enumerate}[(i)]
\item \emph{(Existence and uniqueness)}
For each  $u_0 \in C^\infty_c(S,E)$ and for each $f\in C^\infty_c(M,E)$ there exists a unique $u\in C^\infty(M,E)$ solving $Du=f$ on $M$ and $u=u_0$ along $S$.
\item \emph{(Stability)}
The solution $u$ depends continuously on the data $f$ and $u_0$.
\item \emph{(Finite propagation speed)}
The solution $u$ satisfies $\supp(u) \subset J(K)$ where $K=\supp(u_0)\cup \supp(f)$.
\end{enumerate}
\end{thm}

\begin{proof}
Along $S$ we can write $D=-\sigma\nabla_\mathfrak{n} + Q$ where the coefficient field $\sigma$ is invertible because $\mathfrak{n}$ is timelike and hence non-characteristic ($\sigma$ is Clifford multiplication with $\mathfrak{n}$) and $Q$ is a differential operator involving only derivatives tangential to $S$.
Note that if a section $v$ vanishes along $S$ then $Qv$ vanishes along $S$ as well.

To show existence of solutions let $u_0 \in C^\infty_c(S,E)$ and $f\in C^\infty_c(M,E)$.
We apply Theorem~\ref{thm:cauchyglobhyp} with $P=D^2$ and solve $D^2v=f$ on $M$ with $v|_S=0$ and $\nabla_\mathfrak{n}v=-\sigma^{-1}u_0$ along $S$.
Now we put $u:=Dv$.
Then $Du=D^2v=f$ and along $S$ we have
$$
u = Dv = -\sigma\nabla_\mathfrak{n}v + Qv = -\sigma\nabla_\mathfrak{n}v = u_0
$$
as required.

As to uniqueness, let $Du=0$ on $M$ with $u=0$ along $S$.
Then, clearly, $D^2u=0$ on $M$ and along $S$ we have 
$$
\nabla_\mathfrak{n}u = -\sigma^{-1}(Du-Qu) = -\sigma^{-1}(0-0) = 0.
$$
Thus, by Theorem~\ref{thm:cauchyglobhyp}, $u=0$.
This proves assertion (i).

The other two assertions follow for $u=Dv$ directly from the corresponding statements for $v$ in Theorem~\ref{thm:cauchyglobhyp}.
\end{proof}

Again, it is possible to treat data with less regularity.
If $u_0\in H^k_c(S,E)$ and $f$ is locally $L^2$ in time and $H^k$ in space, then the solution $u$ will be continuous in time and $H^k$ in space.

\subsection{Cauchy problem for the Einstein equation}
The vacuum Einstein equation 
\begin{equation}\label{eq:EVE} 
\mathrm{Ric} = 0
\end{equation} 
is the field equation of general relativity. 
Let $S$ be a smooth spacelike Cauchy hypersurface in $M$.
Then the induced Riemannian metric $h$ on $S$ and the second fundamental form $k$ must satisfy the \emph{constraint equations} 
\begin{subequations}\label{eq:constr}
\begin{align} 
\mathrm{scal}_h + (\tr_h k)^2 - |k|^2 ={}& 0, \\ 
\mathrm{div} (k) - d(\tr_h k) ={}& 0.
\end{align} 
\end{subequations} 
A 3-manifold $S$ with fields $(h,k)$ solving the vacuum constraint equations~\eqref{eq:constr} is called a \emph{vacuum Cauchy data set}. 
A vacuum spacetime $(M,g)$ containing a Cauchy surface $S$ with induced data $(h,k)$ is called a \emph{Cauchy development} of $(S,h,k)$. 
The vacuum Cauchy developments of $(S,h,k)$ can be partially ordered by inclusion and, by Zorn's Lemma, there is thus a \emph{maximal} element of the set of Cauchy developments. 

In a suitable gauge \eqref{eq:EVE} becomes a system of non-linear wave equations. 
Constraint and gauge conditions can be shown to propagate. 
This allows one to solve the Cauchy problem for \eqref{eq:EVE} locally. 
The fact that the field equations are covariant, together with the local uniqueness then implies global uniqueness up to isometry. 

\begin{thm}[{Choquet-Bruhat, Geroch \cite[Thm.~3]{1969CMaPh..14..329C}}] Let $(S,h,k)$ be a vacuum Cauchy data set. Then there is a unique, up to isometry, maximal vacuum Cauchy development $(M,g)$ of $(S,h,k)$.  
\end{thm}


\section{Waves on black-hole spacetimes}

\subsection{Spin geometry in $3+1$ dimensions} \label{sec:spingeom}
An orientable, globally hyperbolic 3+1-dimensional spacetime is spin, so for our purposes, we may assume without loss of generality that $M$ is spin and fix a spin structure $\Pspin\to M$. 
The spin group is $\text{Spin}(1,3) = \text{SL}(2,\Complex)$, the universal covering of the Lorentz group.
It has the inequivalent representations spinor representations $\Complex^2$ and $\bar\Complex^2$. The identification $\Complex \otimes \Reals^4 \equiv \Complex^2 \otimes \bar \Complex^2$ provides a correspondence between tensors and spinors\footnote{Here we use the term spinor to refer to spin-tensors of general valence.}.
%
%
The irreducible representions of $\text{Spin}(1,3)$ are the spaces of symmetric spinors of valence $(k,l)$ where $k,l \in \NatNum_0$. 
We write  $S_{k,l}:=(\Complex^2)^{\odot k}\otimes(\bar\Complex^2)^{\odot l}$ to denote the space of symmetric spinors of valence $(k,l)$. 
Moreover, we use the convention $S_{k,l}=0$ if $k<0$ or $l<0$.

The associated vector bundles will be denoted by 
$$
S_{k,l}M:=\Pspin\times_{\mathrm{Spin}(1,3)}S_{k,l}
$$
and the space of sections of $S_{k,l}M$ by $\mathcal{S}_{k,l}$.

If $k+l$ is even, then sections of $S_{k,l}M$ can be identified with tensor fields. 
Examples are given by vector fields (or $1$-forms), traceless symmetric tensor fields of rank 2, self-dual differential $2$-forms and tensor fields with the symmetries and trace properties of the Weyl tensor.
They correspond to elements of $\mathcal{S}_{k,l}$ with $(k,l) = (1,1), (2,2), (2,0)$ and $(4,0)$, respectively. 

Decomposing spinor and tensor expressions into irreducible pieces yields a powerful tool for simplification and canonicalization. 
The tensor product decomposition 
\begin{equation}\label{eq:tens-prod-decomp}
S_{1,1}\otimes S_{k,l} 
= 
S_{k-1,l-1}\oplus S_{k+1,l-1}\oplus S_{k-1,l+1}\oplus S_{k+1,l+1}
\end{equation} 
yields four \emph{fundamental operators} \cite{ABB:symop:2014CQGra..31m5015A}.
Composing the covariant derivativative $\nabla: C^\infty(M,S_{k,l}M) \to C^\infty(M,T^*M\otimes S_{k,l}M)=C^\infty(M,S_{1,1}M\otimes S_{k,l}M)$ with the corresponding projections gives
\begin{subequations}
\begin{alignat}{3}
\sDiv_{k,l}:{}& \mathcal{S}_{k,l}\rightarrow \mathcal{S}_{k-1,l-1}, \\
\sCurl_{k,l}:{}& \mathcal{S}_{k,l}\rightarrow \mathcal{S}_{k+1,l-1}, \\
\sCurlDagger_{k,l}:{}& \mathcal{S}_{k,l}\rightarrow \mathcal{S}_{k-1,l+1}, \\
\sTwist_{k,l}:{}& \mathcal{S}_{k,l}\rightarrow \mathcal{S}_{k+1,l+1} \, .
\end{alignat}
\end{subequations}
The operators are called the \emph{divergence}, \emph{curl}, \emph{curl-dagger}, and \emph{twistor operators}, respectively. 
Unless otherwise stated, we shall only consider symmetric spinors in the following.

The equations for a massless field of spin $\sfrak \in \half \NatNum$ are  
\begin{subequations} \label{eq:spins-eq}
\begin{align}
\sCurlDagger_{2\sfrak,0} \phi = 0, \quad \phi \in{}& \mathcal{S}_{2\sfrak,0} \quad \text{(left handed fields)} \label{eq:spins-eq-left}, \\ 
\sCurl_{0,2\sfrak} \varphi = 0, \quad \varphi \in{}& \mathcal{S}_{0,2\sfrak} \quad \text{(right handed fields)} \label{eq:spins-eq-right}.
\end{align}
\end{subequations}
For $\sfrak=1/2$ equation \eqref{eq:spins-eq-left} is the Dirac equation for massless left-handed spinors, also known as Dirac-Weyl equation, and in case $\sfrak=1$ \eqref{eq:spins-eq} are the left and right Maxwell equations. The Maxwell equation for a real field strenght $F_{ab}$ is equivalent to either of the equations \eqref{eq:spins-eq}.

The irreducible components of the Riemann curvature tensor are the scalar curvature, the traceless Ricci tensor and the Weyl tensor. The Weyl tensor corresponts to the Weyl spinor $\Psi \in \mathcal{S}_{4,0}$.
The vacuum Einstein equations imply that the Weyl spinor satisfies the spin-2, or Bianchi equation\footnote{One conventionally chooses the left-handed Weyl spinor $\Psi \in \mathcal{S}_{4,0}$ to represent the real Weyl tensor.}
\begin{equation}\label{eq:bianchi}
\sCurlDagger_{4,0} \Psi = 0 .
\end{equation} 
For $\sfrak \geq 3/2$, the existence of a non-trivial solution to the spin-$\sfrak$ equation implies  curvature conditions, a fact known as the Buchdahl constraint \cite{Buchdahl58}. 
For example, the spin-2 field on a general background is proportional to the Weyl spinor on the background.

A spinor $\varkappa$ of valence $(k,l)$ satisfying 
\begin{equation}\label{eq:killspin}
\sTwist_{k,l} \varkappa = 0
\end{equation}
is called a \emph{Killing spinor}. Killing spinors of valence $(1,1)$, $(2,0)$ and $(2,2)$ correspond to conformal Killing fields, conformal Killing-Yano 2-forms, and traceless conformal Killing tensors of rank 2, respectively. Recall that a conformal Killing-Yano 2-form is a tensor $Y_{ab} = Y_{[ab]}$, satisfying 
\begin{align}
\label{eq:CKYdefTensorVersion}
\nabla_{(a}Y_{b)c}={}&- \tfrac{1}{3} g_{ab} \nabla_{d}Y_{c}{}^{d}
 + \tfrac{1}{3} g_{(a|c|}\nabla^{d}Y_{b)d}.
\end{align}
Similarly, a Killing tensor is a tensor field $K_{ab \cdots d} = K_{(ab\cdots d)}$ satisfying $\nabla_{(a} K_{bc\cdots d)} = 0$. 

The Petrov classification gives a classification of spacetimes according to the algebraic type of the Weyl spinor. The following result is a consequence of the Bianchi identity. 

\begin{thm}[{Walker, Penrose \cite[Lemma~1]{walker:penrose:1970CMaPh..18..265W}}] \label{thm:Walker-Penrose}
Assume $(M, g)$ is a vacuum spacetime of Petrov type D. Then $(M, g)$ admits a one-dimensional space of Killing spinors of valence $(2,0)$.  
\end{thm} 
As mentioned above, a Killing spinor of valence $(2,0)$ is equivalent to a real conformal Killing-Yano 2-form. 
The existence of a Killing spinor implies symmetry operators for the massless field equations, as well as conservation laws. By a \emph{symmetry operator}, we mean an operator which takes solutions to solutions. 

Given a Killing spinor $\kappa \in \mathcal{S}_{2,0}$ we may define a complex scalar field $\kappa_1$ by choosing $\kappa_1^2$ to be proportional to the $\mathcal{S}_{0,0}$ term in the expansion of $\kappa \otimes \kappa \in \mathcal{S}_{2,0} \otimes \mathcal{S}_{2,0}$ into irreducible components, analogous to  \eqref{eq:tens-prod-decomp}, and choosing the principal root to define $\kappa_1$. If $(M,g)$ is of Petrov type D, then $\kappa_1$ defined in this manner is non-zero.

Let 
\begin{equation}\label{eq:Udef}
U_a = - \nabla_a \log (\kappa_1) ,
\end{equation}
and define, using the same procedure as discussed above, the extended fundamental operators $\sDiv_{k,l,m,n}, \sCurl_{k,l,m,n}, \sCurlDagger_{k,l,m,n}, \sTwist_{k,l,m,n}$ by projections of $\nabla_{a} + m U_{a} + n \bar U_a$; see \cite[\S 2.2]{2016arXiv160106084A} for details. Here and below, $\bar{\ }$ denotes complex conjugation. 

Finally, given a Killing spinor $\kappa \in \mathcal{S}_{2,0}$ in a Petrov D spacetime, we define for $\varphi \in \mathcal{S}_{k,l}$,  the operators $\varphi \mapsto \mathcal{K}^j_{k,l} \varphi$, $j=0,1,2$ as the $\mathcal{S}_{k+2-2j,l}$ components in the irreducible decomposition of $\kappa_1^{-1} \kappa \otimes \varphi$. Here we restrict to parameters such that $k+2-2j \geq 0$. The operators $\bar{\mathcal{K}}^j_{k,l}$ are defined analogously in terms of projections of $\bar \kappa_1^{-1} \bar \kappa \otimes \varphi$. 

We empasize that symbolic computations have played an essential role in the results on fields with non-zero spin presented here, in particular in sections \ref{sec:struct}, \ref{sec:TMETSI}. The packages \emph{SymManipulator} \cite{Bae11a} and \emph{SpinFrames} \cite{AkBae15}, developed for the Mathematica based symbolic differential geometry suite \emph{xAct} \cite{xAct} makes it possible to systematically exploit decompositions in terms of irreducible representations of the spin group $\text{Spin}(1,3)$, and allows one to carry out investigations that are not feasible by hand.

\subsection{Geometry of black hole spacetimes}
A spacetime $(M,g)$ is \emph{asymptotically flat at null and spatial infinity} if there is a spacetime $(\tilde M, \tilde g)$ which is $C^\infty$ except possibly at a point $i^0$ called \emph{spatial infinity}, and a conformal diffeomorphism $\psi : M \to \psi(M) \subset \tilde M$, 
$$
\tilde g = \Omega^2 \psi^* g, \quad \text{ in $\psi(M)$}
$$
with conformal factor $\Omega$, such that  
$$
\overline{J^+(i^0)} \cup \overline{J^-(i^0)} = \tilde M - M
$$
Here we have, for simplicity, identified $M$ with $\psi(M)$. The boundaries of the causal future and past of $i^0$ are called null infinity and denoted $\Scri^{\pm}$, 
$$
\Scri^{\pm} = \partial J^{\pm}(i^0) \setminus i^0 \, .
$$
We further require that $\Omega = 0$ and $\tilde \nabla \Omega \ne 0$ on $\Scri^{\pm}$; see \cite[\S 11.1]{MR757180} for details. If in addition, there is a neighborhood $\tilde V \subset \tilde M$ of $\overline{M \cap J^-(\Scri^+)}$ which is globally hyperbolic, then $M$ is said to be \emph{strongly asymptotically predictable}. 

Given the notions introduced above, we say that $(M,g)$ is a \emph{black hole spacetime} if it is asympotically flat at null and spatial infinity, strongly asymptotically predicable, and such that
$M \setminus J^{-}(\Scri^+) \ne \emptyset$. In this case, $H = \partial J^-(\Scri^+) \cap M$ is called the \emph{event horizon}, and $O = I^+(\Scri^-) \cap I^-(\Scri^+)$ is called the \emph{domain of outer communication} (DOC).  

An example of these notions is provided by the Schwarzschild spacetime, i.e. the unique static, spherically symmetric black hole spacetime satisfying the Einstein vacuum equations 
$$
\mathrm{Ric} = 0 .
$$
In Schwarzschild coordinates $(t,r,\theta, \phi)$, the Schwarzschild metric takes the form
\begin{equation}\label{eq:g_Schw}
g =  - f dt^2 + f^{-1} dr^2 + r^2 d\Omega^2_{S^2}
\end{equation}
with $f = 1-2\mass/r$. 
Here $d\Omega^2_{S^2} = d\theta^2 + \sin^2\theta d\phi^2$ is the standard metric on the unit 2-sphere. 
An explicit construction of the conformal diffeomorphism $\psi$ and conformal factor $\Omega$ satisfying the above conditions for the Schwarzschild spacetime is provided by the Kruskal-Szekeres extension.
\begin{figure}[!h]
\centering
\begin{tikzpicture}[scale=.7,decoration={snake,amplitude=.5mm,segment length=2mm}]

\draw[thick, blue] (-4,4) -- (4,-4);
\draw[thick,blue] (-4,-4) -- (4,4) node[near end,rotate=45,yshift=.2cm]{$r=2\mass$};


\draw[gray] (4.5,2) node{$\mathrm{I}$};

\draw[gray] (0,2) node{$\mathrm{II}$};

\draw[gray] (-4.5,2) node{$\mathrm{III}$};

\draw[gray] (0,-2.5) node{$\mathrm{IV}$};

\draw[thick, dash dot, red](4,4)--(8,0) node[midway,xshift=.4cm,yshift=.4cm]{$\Scri^+$};
\draw[thick, dash dot, red](4,-4)--(8,0)node[midway,xshift=.4cm,yshift=-.4cm]{$\Scri^-$};

\draw[thick, dash dot, red](-4,4)--(-8,0);
\draw[thick, dash dot, red](-4,-4)--(-8,0);


\draw[very thin, black, ->] (-1,.3) node[anchor=east]{$\mathcal{B}$} to[bend right=20] (-.2,0);

\filldraw[fill=black] (0,0) circle[thick,radius=.7mm] ;

\draw[decorate,very thick, red](-4,4)--(4,4);
\draw[decorate,very thick, red](-4,-4)--(4,-4);

\draw[thin] (0,0) .. controls (4,1) .. (8,0);
\draw[thin] (0,0) .. controls (4,0) .. (8,0);
\draw[thin] (0,0) .. controls (4,-1) .. (8,0);


\draw[thick,nearly transparent] (4,-4) .. controls(2,0) .. (4,4) node[near end,xshift=.3cm,rotate=65] {$r=3\mass$};


\draw[thin,<-]  (4.5,-1) to [bend right] (6,-5) node[anchor=west]{$\{t=\text{constant}\}$};


\filldraw[fill=white] (8,0) circle[thick,radius=.7mm] node[anchor=west]{$i_0$} ;

\filldraw[fill=white] (4,4) circle[thick,radius=.7mm] node[anchor=west,yshift=.3cm]{$i_+$};

\filldraw[fill=white] (4,-4) circle[thick,radius=.7mm] node[anchor=west,yshift=-.3cm]{$i_-$};

\filldraw[fill=white] (-4,4) circle[thick,radius=.7mm] ;

\filldraw[fill=white] (-8,0) circle[thick,radius=.7mm] ;

\filldraw[fill=white] (-4,-4) circle[thick,radius=.7mm] ;

\end{tikzpicture} 
\caption{The maximal extension of the Schwarzschild spacetime.
Region I is the domain of outer communication.}
\label{fig:krusk}

\end{figure}
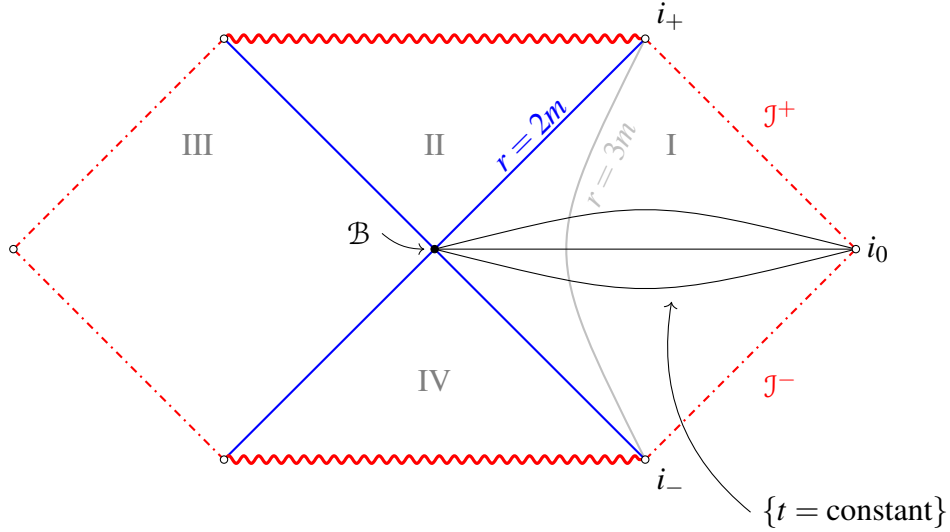
Referring to figure \ref{fig:krusk}, the DOC is represented by region I.  The parameter $\mass$ is the ADM mass of the Schwarzschild black hole. 

\subsection{The Kerr solution}
The Kerr family of vacuum, rotating black hole spacetimes, provides the most important example for our purposes. In Boyer-Lindquist coordinates $(t,r,\theta,\phi)$, the Kerr metric takes the form 
\begin{align}
g \nonumber ={}&{}
\left(-1+\frac{2\mass r}{\Sigma}\right)d t^2 -\frac{4\mass ra\sin^2\theta}{\Sigma}
d t d\phi 
\nonumber\\ {}&{}+
\frac{\Pi\sin^2\theta}{\Sigma} d \phi^2  +\frac{\Sigma}{\Delta}d r^2
 +\Sigma d \theta^2 ,
\label{eq:KerrMetric}
\end{align}
where
\begin{align*}
\Sigma &= r^2 + a^2\cos^2\theta, \\
\Delta &= r^2-2mr+a^2, \\
\Pi    &= (r^2+a^2)^2 -\Delta a^2\sin^2\theta .
\end{align*}

The Kerr spacetime admits two Killing vector fields, the stationary Killing field $\partial_t$ which is timelike near infinity, and the axial $\partial_\phi$. The Kerr family of metrics is parametrized by the mass $\mass$ and the angular momentum per unit mass $a$. The mass and angular momentum $a\mass$ are the ADM momenta corresponding to $\partial_t$ and $\partial_\phi$ respectively. In the subextreme case $|a| < \mass$, the Kerr spacetime contains a non-degenerate black hole\footnote{A stationary black hole is said to be \emph{non-degenerate} if the surface gravity $\upkappa$ of the event horizon is non-zero, where $\upkappa$ is defined by $\upkappa^2 = -\half (\nabla^a \chi^b)(\nabla_a \chi_b)$. Here $\chi^a$ is a null generator of the event horizon.}.

The Kerr metric is algebraically special, of Petrov type D. In particular, a principal null tetrad $\NPl^a, \NPn^a, \NPm^a, \NPmbar^a$ normalized so that $g_{ab} = 2( \NPl_{(a} \NPn_{b)} - \NPm_{(a} \NPmbar_{b)})$ can be chosen, where $\NPl^a$, $\NPn^a$ are repeated principal null directions. 

By Theorem \ref{thm:Walker-Penrose}, the Kerr spacetime admits a Killing spinor $\kappa$ of valence $(2,0)$, or equivalently a conformal Killing-Yano 2-form. Associated with $Y_{ab}$ is the complex $1$-form $\xi = \sCurlDagger \kappa$, or in terms of $Y_{ab}$,
\begin{align}
\xi_{a}= 
{}&
\tfrac{1}{3}i \nabla_{b}Y_{a}{}^{b}
 -  \tfrac{1}{3} \nabla_{b}{*}Y_{a}{}^{b}.
\label{eq:xidefTensorVersion}
\end{align}
In the case of Kerr, or more generally in the \emph{generalized Kerr-NUT class of type D spacetimes}, $Y_{ab}$ can be chosen so that $\nabla^a Y_{ab} = 0$, in which case $Y_{ab}$ is a Killing-Yano 2-form satisfying 
\begin{equation}\label{eq:KY}
\nabla_{(a} Y_{b)c} = 0. 
\end{equation} 
In particular, in this case $\xi_a$ is real. It is convenient to normalize $Y_{ab}$ (or equivalently $\kappa$) so that in Kerr, in terms of Boyer-Lindquist coordinates, $\xi^a = (\partial_t)^a$, and $\kappa_1 = -\tfrac{1}{3}(r-ia\cos\theta)$. 

If \eqref{eq:KY} holds, the square of $Y_{ab}$,  $K_{ab} = Y_{ac} Y^{c}{}_b$ is a symmetric Killing tensor, 
$$
\nabla_{(a} K_{bc)} = 0. 
$$
The Killing tensor $K_{ab}$ is the Carter tensor in the Kerr spacetime. 

In a general spacetime with only two Killing vector fields, the geodesic motion is chaotic. The conserved quantities for geodesics in Kerr associated to $\partial_t$ and $\partial_\phi$ are energy $\GeodesicEnergy = \dot \gamma^a (\partial_t)_a$ and azimuthal angular momentum $\GeodesicLz = \dot \gamma^a (\partial_\phi)_a$. However, the presence of the Carter Killing tensor provides an additional conserved quantity $\KCarter = K_{ab} \dot \gamma^a \dot \gamma^b$ for geodesics in the Kerr spacetime, the Carter constant. Hence, by the Liouville theorem the geodesic equations in Kerr can be integrated by quadratures. Further, as will be discussed below, it provides separability and decoupling properties which allow one to analyze fields on the Kerr spacetime.


\section{Fields on Kerr spacetime}

A necessary prerequisite for solving the black hole stability problem is to be able to prove decay estimates for linear test fields on the Kerr background. The fields of interest for the black hole stability problem are massless test fields of spins up to 2. 

The field equations for massless test fields of integer spin $\sfrak = 0,1,2$ are 
\bigskip

\begin{itemize} \addtolength{\itemsep}{0.1in}
\item[\ ] $(\sfrak=0)$ \quad  Scalar wave equation 
$$
\nabla^a \nabla_a \phi = 0, 
$$
\item[\ ] $(\sfrak=1)$ \quad  Maxwell equation 
$$
F_{ab} = F_{[ab]}, \quad \nabla^a F_{ab} = 0, \quad \nabla_{[a} F_{bc]} = 0,
$$
\item[\ ] $(\sfrak=2)$ \quad Linearized gravity 
$$
D\mathrm{Ric}(\delta g) = 0.
$$ 
\end{itemize} 
Here $D\mathrm{Ric}$ is the linearization of the Ricci tensor at the Kerr background. 
Let $g(\lambda)$ be a 1-parameter family of metrics with $g(0)$ the Kerr metric and $\frac{d}{d\lambda} \big{|}_{\lambda=0} g(\lambda)=\delta g$. 
Then $D\mathrm{Ric}$ is defined as 
$$
D\mathrm{Ric}(\delta g) = \frac{d}{d\lambda} \bigg{|}_{\lambda=0} \mathrm{Ric}[g(\lambda)].
$$

In contrast to the scalar field, the field equations of non-zero spin admit bound states, i.e. non-trivial time independent solutions, on the Kerr spacetime. 
We refer to these as non-radiating modes. For the Maxwell field, the non-radiating modes are the time independent Coloumb solutions, and for linearized gravity, the non-radiating modes correspond to deformations within the Kerr family. Thus, for a field of non-zero integer spin, a local integrated energy estimate must be independent of the non-radiating modes, which constitutes a significant additional difficulty compared to the spin-0 case. 


\subsection{Integrated energy decay}

The essential step needed for a proof of decay estimates for massless fields is to prove dispersion. 
By dispersion we mean that the energy density of the field in a local region decays in time. 
Dispersion is expressed by an integrated local energy estimate, or Morawetz estimate, of the form\footnote{We say that $a \lesssim b$ if there is a suitably universal constant $C > 0$ so that $a \leq C b$.}  
$$
\int_{t_0}^{t_1} \int_{R} W \lesssim I
$$
where $W$ dominates the energy density of some order $k$, up to a weight depending on $r$ (a degeneration of $W$ near trapping is acceptable) and $I$ depends only on an energy of order $k$ evaluated on initial data. 
For a non-linear stability result, pointwise decay estimates are essential. 
Once a Morawetz estimate is available, such estimates can be proved using standard techniques. We will here only discuss Morawetz estimates.

Two features of a black hole spacetime are constitute important obstacles to proving such a result, namely trapping and superradiance.  
In a black hole spacetime, there are future directed null geodesics emanating from every point which either cross the horizon or null infinity. By continuity, there must therefore be geodesics which do neither, and hence there are trapped null geodesics which orbit the black hole for all time. In the Schwarzschild spacetime these are located at $r=3\mass$, and in Kerr the location of the trapped null geodesic depends on the conserved quantities $\GeodesicEnergy, \GeodesicLz, \KCarter$. 

In view of the geometric optics approximation,  there are high frequency wave packets near these trapped orbits tend to disperse slowly, and hence trapping is an obstacle to dispersion. 

In the Kerr spacetime, the stationary Killing vector field $\partial_t$ which is timelike near infinity becomes spacelike near the horizon. The region where this happens is called the ergoregion. Due to the existence of the ergoregion, there is no positive definite conserved energy for fields in the Kerr spacetime. One may show that there are wave packets whose energy measured near infinity increases, when scattered off the black hole. 

Finally, the group of isometries of Kerr is only two-dimensional, which is an obstacle to constructing a sufficient number of conserved currents based on Lie-symmetries of the spacetime, which can be used to control fields on Kerr. The existence of the Carter Killing tensor and related conservation laws and symmetry operators (which may be termed \emph{hidden symmetries}) can be used to circumvent the problems caused by the lack of Lie-symmetries. 

As mentioned above, the existence of the Carter Killing tensor is deeply related to separability properties of the field equations. The equations for massless fields on the Kerr spacetime admit a second order symmetry operator related to the Carter Killing tensor, which are not reducible to the first order Lie symmetries $\Lie_{\partial_t}, \Lie_{\partial_\phi}$. 

Define $\OpQ$ by
\begin{equation}\label{eq:OpQdef}
\OpQ = \frac1{\sin\theta}\dtheta\sin\theta\dtheta
+\frac{\cos^2\theta}{\sin^2\theta}\dphi^2 +a^2\sin^2\theta \dt^2 .
\end{equation} 
The Carter symmetry operator which commutes with $\nabla^a \nabla_a$ involves derivatives with respect to both $r,\theta$. 
The operator $\OpQ$ is a symmetry operator for $\nabla^a \nabla_a$ but does not commute; instead $\OpQ$ commutes with $\KSigma \nabla^a \nabla_a$. Here $\KSigma$ may be viewed as a separating factor. 

We denote 
the set of order-$n$ generators of the symmetry algebra 
generated by $\dt$, $\dphi$, and $\OpQ$ by
\begin{align}
\SymGeneraln{n}= \{ \dt^\nt \dphi^\np \OpQ^\nq \mid \nt +\np +2\nq = n; \nt,\np,\nq\in\Naturals \} .
\label{eq:KerrHigherSymmetries}
\end{align}
In particular, 
$$
\SymGeneraln{0}=\{\Id\} , \quad \SymGeneraln{1}=\{\dt,\dphi\} . 
$$
Of particular importance in our analysis will be the set of second-order symmetry operators, 
\begin{align*}
\SOSym =\{ \dt^2, \dt\dphi, \dphi^2, \OpQ \}=\{\CQA_\ua \} ,
\end{align*}
and 
underlined
indices always refer to the index in this set.

A key observation is that one may use generalized vector fields 
$$
A^a = \FF^{a\ua\ub} \partial_a S_{\ua} S_{\ub}, \quad S_{\ua} \in \Symop_2,
$$
to define currents using the formula 
$$
J_a = T_{ab\ua\ub} A^b
$$
where  for a field $\phi$, $T_{ab\ua\ub}$ is defined via a polarization of the standard stress-energy tensor $T_{ab}$, 
$$
T_{ab\ua\ub}[S_{\ua} \phi, S_{\ub} \phi] = \tfrac{1}{4} ( T_{ab}[S_{\ua} \phi + S_{\ub} \phi] - T_{ab}[S_{\ua} \phi - S_{\ub} \phi] \, .
$$
This approach yields a \emph{local spacetime proof} of a Morawetz estimate for the wave equation on slowly rotating Kerr black holes with $|a| \ll \mass$.

Higher-order pointwise norms are defined 
in terms of $\SymGeneraln{n}$
by 
\begin{equation}\label{eq:ptwise} 
\normPtwiseTn{n}{\solu}^2=
\sum_{j=0}^n\sum_{\CQA\in\SymGeneraln{j}} |\CQA\solu|^2. 
\end{equation} 

We now state our main results and briefly compare them with previous results.
In formulating our estimates, we shall make use of the following model energy,
\begin{align*} 
\ModelEnergyThree[\solu](\hst{t})\\
=\int_{\hst{t}}{}&{} \left(\frac{(r^2+a^2)^2}{\KDelta}|\dt\solu|_2^2 +\KDelta|\dr\solu|_2^2+|\dtheta\solu|_2^2+\frac{1}{\sin^2\theta}|\dphi\solu|_2^2\right)\diThree , 
\end{align*}
where $|\cdot|_2$ is the second-order point-wise norm introduced in \eqref{eq:ptwise} above, and 
$
\diThree=\sin\theta\di r\di\theta\di\phi  
$
is a reference volume element on the Cauchy slice $\hst{t}$. 

To deal with the necessary degeneracy in the Morawetz estimate near trapping, we introduce a cutoff function $\localiseAwayFromPhotonOrbits$ which is supported away from trapping. 

\begin{theorem}[{Andersson, Blue \cite[Thm.~1.2]{AnderssonBlue:KerrWave}}]
\label{Thm:IntroMorawetz}
Let $\solu$ be a solution of the wave equation on the Kerr exterior spacetime, 
$\nabla^a\nabla_a\solu=0$. Then 
\begin{align*}
\int_{-\infty}^{\infty} \int_{\rp}^{\infty} \int_{S^2}
 \bigg( \left(\frac{\KDelta^2}{r^4}\right) \normPtwiseTn{2}{\dr \solu}^2 +r^{-2}\normPtwiseTn{2}{\solu}^2  
+ \localiseAwayFromPhotonOrbits \frac{1}{r} \left(\normPtwiseTn{2}{\dt \solu}^2+\normPtwiseTn{2}{\dAng \solu}^2\right) \bigg)
\diFour  \\
\lesssim \ModelEnergyThree(\hst{0}) .
\end{align*}
\end{theorem}

Recently, a Morawetz and pointwise decay estimates have been proved for the scalar field equation on Kerr for the whole subextreme range $|a| < \mass$ has been given, using Fourier techniques and separation of variables \cite{2014arXiv1402.7034D}. The problem of giving a local, spacetime based proof of such a result is open.

For massless fields with non-zero spin, any Morawetz estimate must eliminate the non-radiating modes, in particular the spacetime energy density $W$ must cancel the non-radiating mode. For the spin-1 or Maxwell case, the best result so far is the following decay estimate for a test Maxwell field on a slowly rotating Kerr background. 

To state this result, we define, for a charge-free Maxwell field
with components $\phii$ and $\Upsilon=\Upsilon[\MaxF]=(r-ia\cos\theta)\phi_1$, the following bulk space-time integrals
\begin{subequations}
\begin{align}
\BulkMorawetzF
&=\intBulk{\infty}\frac{\mass\KDelta}{(r^2+a^2)^2} 
(|\phi_0|^2 + |\phi_2|^2)  
\diFourNatural ,\\
\BulkFILowerOrder
&=\intBulk{\infty} \frac{\mass|\phi_1|^2}{r^2}\diFourNatural  
=\intBulk{\infty} \frac{\mass}{r^4}|\Upsilon|^2r^2\diFourFI ,\\
\BulkFITwo
&=\intBulk{\infty} \left(\frac{\mass\KDelta^2}{(r^2+a^2)r^2}|\pr\Upsilon|^2 +\localiseAwayFromPhotonOrbits\frac{\mass^2|\pt \Upsilon|^2+|\pAng\Upsilon|^2}{r}\right)\diFourFI, \\
\BulkFIOne
&=\int_{\rp}^{\infty} (1-\localiseAwayFromPhotonOrbits)
\left|\int_{0}^{\FinalTime}\int_{\Stwo} \Im(\solub\pt\Upsilon) \diomega\di t\right| \di r ,
\end{align}
\end{subequations}
where $\diFourNatural$ is the geometrically defined volume form $\KSigma\diomega\di r\di t$, and $\diFourFI$ is the coordinate volume form $\diomega\di r\di t$. For $T<0$, we reverse the sign in these bulk terms, so that they remain nonnegative. The indexing is chosen so that $\BulkMorawetzF$ involves $\phi_0, \phi_2$, and $\BulkFILowerOrder$ involves $\phi_1$ or, equivalently, $\Upsilon$ with no derivatives, $\BulkFIOne$ involves $\Upsilon$ with one derivative in the integral, and $\BulkFITwo$ involves $\Upsilon$ with two derivatives but with a degeneracy near the orbiting null geodesics.

\begin{theorem}[{Andersson, Blue \cite[Thm.~1.3]{andersson:blue:maxwell:2013arXiv1310.2664A}}]\label{thm:Morawetz}
There are positive constants $\epsilonSlowRotationIntro$, $\ConstantIntro$, such
that 
if 
$\epsilonMain\leq\epsilonSlowRotationIntro$, 
$\MaxFTotal$ is a regular solution of the Maxwell equation 
for which $\EnergyCrudeFArg{\MaxFTotal}(0)$ and $\EnergyCrudeFIArg{\soluTotal}(0)$ are finite, 
and the quantities $\BulkMorawetzF$, $\BulkFILowerOrder$, $\BulkFIOne$, and $\BulkFITwo$ are defined in terms of the uncharged part $\MaxF$, 
then $\forall T\in\Reals:$
\begin{align}
\label{eq:BulkTermsDef}
\BulkMorawetzF
+\BulkFILowerOrder
+\BulkFIOne
+\BulkFITwo
&\leq \ConstantIntro\left(\EnergyCrudeFArg{\MaxFTotal}(0) + \EnergyCrudeFIArg{\soluTotal}(0)\right)   .
\end{align}
\end{theorem}

Although this type of decay estimate is relatively weak, it is sufficiently robust to have formed the foundation for many further decay results in the study of fields outside black holes. 

\subsection{Structure of massless field equations on Kerr} \label{sec:struct}

Let $(M,g)$ be a vacuum spacetime with a valence $(2,0)$ Killing spinor $\kappa$. 
Let $\xi = \sCurlDagger_{2,0} \kappa$, and assume further that $\xi$ is real. This holds in particular in the Kerr spacetime for a suitable normalization of $\kappa$. 

By a \emph{symmetry operator} for a field equation, we mean an operator which takes solutions to solutions. 
As shown by Carter, a rank 2 Killing tensor $K_{ab}$ in a vacuum spacetime provides a \emph{commuting} symmetry operator for the wave equation 
$$
[\nabla_a K^{ab} \nabla_b, \nabla^a \nabla_a] = 0 .
$$
As we shall now discuss, the Maxwell equation on the Kerr spacetime (and more generally on vacuum spacetimes of type D) admits symmetry operators. 

Recall that we defined the extended fundamental spinor operators $\sDiv_{k,l,m,n}$, $\sCurl_{k,l,m,n}$, $\sCurlDagger_{k,l,m,n}$, $\sTwist_{k,l,m,n}$ as well as the operators $\mathcal{K}^j_{k,l}$, $\bar{\mathcal{K}}^j_{k,l}$ in section \ref{sec:spingeom}. 

\begin{definition} 
Define the first-order 1-form linear concomitants $A$ and $B$ by 
\begin{subequations}\label{eq:ABdef} 
\begin{align}
A[\kappa,\phi]={}&\overline{\mathcal{K}}^1_{1,1}\sCurlDagger_{2,0,0,2} \mathcal{K}^1_{2,0}(\kappa_1 \bar{\kappa}_{1'}\phi), \\
B[\kappa,\phi]={}& \sCurlDagger_{2,0,2,0} \mathcal{K}^1_{2,0}\mathcal{K}^1_{2,0}(\kappa_1^2\phi). 
\end{align} 
\end{subequations} 
\end{definition} 
When there is no room for confusion, we suppress the arguments, and write simply $A$ and $B$. 
The following result shows that $A, B$ solves the adjoint Maxwell equations, 
provided $\phi$ solves the Maxwell equation. 

\begin{theorem}[Andersson, B\"ackdahl, Blue \protect{\cite[Prop.~5.3]{2015arXiv150402069A}}] 
\label{lem:ABpot}
Assume that $\kappa$ is a Killing spinor of valence $(2,0)$ and that $\phi$ is a Maxwell field. Then, with $A, B$ given by \eqref{eq:ABdef} 
the spinors $\chi$ and $\omega$, defined by 
\begin{align}
\chi={}&Q \phi+\sCurl_{1,1} A, \label{eq:SymFirstPot} \\
\omega={}&\sCurlDagger_{1,1} B, \label{eq:SymSecondPot} 
\end{align}
are solutions of the left and right Maxwell equations~\eqref{eq:spins-eq}, respectively. 
\end{theorem} 

The energy-momentum tensor for the Maxwell field, $T = \phi \bar \phi$, or in terms of the field strength tensor $F_{ab}$, 
$$
T_{ab} = F_{ac} F_{b}{}^c - \frac{1}{4} F_{cd} F^{cd} g_{ab} 
$$
is conserved and traceless. Hence, if $\GenVec^a$ is a conformal Killing vector field, $J_a = T_{ab} \GenVec^a$ is conserved. However, in a spacetime with hidden symmetry in the form of a Killing spinor, there are higher order conserved tensors for the Maxwell field. 

If $\GenVec^a$ is a conformal Killing field, then one may define a
first order operator $\varphi \to \hat{\mathcal{L}}_\GenVec \varphi$ 
providing a lift of the Lie-derivative. The operator $\mathcal{L}_\GenVec$ defines a symmetry operator of first order. For the equations of spins $0$ and $1$, the only first order symmetry operators are given by conformal Killing fields.

Consider the operator $F_{ab} \to Z_{ab}$ defined by 
\begin{equation}
Z_{ab}=- 
\tfrac{4}{3} ({*}F)_{[a}{}^{c}Y_{b]c},\label{eq:thetadefTensorVersion}
\end{equation}
This is proportional to $\kappa_1 \mathcal{K}^1_{2,0}$. 

\begin{theorem}[{Andersson, B\"ackdahl, Blue \cite[Thm.~1.1]{2014arXiv1412.2960A}}] 
\label{thm:MainResultTensorVersion} Let $(M, g)$ be a vacuum spacetime of dimension 4, and let $Y_{ab}$ and $F_{ab}$ be real $2$-forms. 
Define the complex
$1$-form $\eta_a$ by
\begin{align}
\eta_{a}={}& \nabla_{b}Z_{a}{}^{b}
 + i \nabla_{b}{*}Z_{a}{}^{b},
\label{eq:etadefTensorVersion}
\end{align}
where $Z_{ab}$ is given by \eqref{eq:thetadefTensorVersion},
and the
real symmetric $2$-tensor $V_{ab}$ by 
\begin{align}
V_{ab}={}&\eta_{(a}\bar{\eta}_{b)}
- \tfrac{1}{2} g_{ab} \eta^{c} \bar{\eta}_{c}
-  \tfrac{1}{3} (\mathcal{L}_{Re\xi}F)_{(a}{}^{c}Z_{b)c}
 + \tfrac{1}{12} g_{ab} (\mathcal{L}_{Re\xi}F)^{cd} Z_{cd}\label{eq:VdefTensorVersion}\\
& + \tfrac{1}{3} (\mathcal{L}_{Im\xi}{*}F)_{(a}{}^{c}Z_{b)c}
 -  \tfrac{1}{12} g_{ab} (\mathcal{L}_{Im\xi}{*}F)^{cd}
 Z_{cd},\nonumber
\end{align}
where $\xi_a$ is given by equation
\eqref{eq:xidefTensorVersion}.

If $Y_{ab}$ is a conformal Killing-Yano tensor and $F_{ab}$ satisfies the
Maxwell equations,  
then $V_{ab}$ is conserved, 
$$
\nabla^a V_{ab} = 0.
$$ 
\end{theorem}

The leading order part of the conserved tensor $V_{ab}$ satisfies the dominant energy condition, and hence one may use $V_{ab}$ to construct energy currents which are positive definite to leading order. This can be expected to yield an approach to decay estimates for the Maxwell field on the Kerr spacetime which is more systematic than the approach used in the proof of theorem \ref{thm:Morawetz}, which relied on a Fourier based Morawetz estimate for the wave equation for the middle component of the Maxwell field. 
 
An examination of the proof of Theorem~\ref{thm:MainResultTensorVersion} shows that the fact that $V_{ab}$ is conserved follows from the Teukolsky Master Equation (TME) and the Teukolsky-Starobinsky Identities (TSI), which are integrability conditions implied by the spin-$\sfrak$ field equations. 
In view of this fact, a deeper understanding of the TME and TSI systems for the spin-2 or linearized gravity case is fundamental in order to generalize conservation laws of the type exhibited in Theorem~\ref{thm:MainResultTensorVersion} to the spin-2 case. 
We shall now briefly mention recent work which provides the initial step in this direction.

\subsection{TME and TSI for linearized gravity} \label{sec:TMETSI}
Let $\delta g$ be a solution to the linearized vacuum Einstein equation on the Kerr background and let $\kappa \in \mathcal{S}_{2,0}$ be the Killing spinor of valence $(2,0)$. Let $\dot \Psi$ be the linearized Weyl spinor defined with respect to $\delta g$, and let 
$$
\phihat = \mathcal{K}^1_{4,0} \dot \Psi, 
$$
where $\mathcal{K}^1_{4,0}$ was introduced in section \ref{sec:spingeom}. 
Defining $\MidMet \in \mathcal{S}_{2,2}$ by 
\begin{align*}
\MidMet= \sCurlDagger_{3,1} \sCurlDagger_{4,0,4,0}  \phihat , 
\end{align*} 
one finds that the traceless symmetric rank 2 tensor $\MidMet$ is a \emph{complex} solution of the linearized Einstein equation.  
The linearized metric $\MidMet$ is generated using $\phihat$ as a Hertz potential.
\begin{lemma}[Aksteiner, Andersson, B\"ackdahl {\cite[Theorem~1.1]{2016arXiv160106084A}}] 
There is a complex vector field $A$ so that 
\begin{align} \label{eq:MidMetgauge}
\MidMet={}& 
\tfrac{\mass}{27}\Lie_{\xi} \delta g
+ 
\Lie_A g .
\end{align}
\end{lemma}

The fact that $\MidMet$ is pure gauge apart from the term involving $\Lie_\xi \delta g$ yields, after applying $(\sCurlDagger)^2$ to both sides of \eqref{eq:MidMetgauge}, and recalling that $(\sCurlDagger)^2$ is the complex conjugate of the map to linearized curvature, yields the following result. 
\begin{theorem}[Aksteiner, Andersson, B\"ackdahl {\cite[Cor.~1.4]{2016arXiv160106084A}}]
\begin{align*} 
\sCurlDagger_{1,3}\sCurlDagger_{2,2}\sCurlDagger_{3,1}
\sCurlDagger_{4,0,4,0} \phihat ={}& 
\tfrac{\mass}{27} \Lie_\xi \overline{\phi} 
+ \Lie_A \overline{\Psi} .
\end{align*}
\end{theorem} 
This is the full, covariant form of the TSI for linearized gravity. 


\section{Index theorem for Dirac operators on Lorentzian manifolds}

Index theory is a huge and well-developed field in the Riemannian setting where Dirac operators are elliptic.
That the hyperbolic Dirac operator on a Lorentzian spin manifold can also be Fredholm under suitable assumptions and hence possess an index is a rather recent insight. 
In fact, applications in physics demand such an index formula; in the past physicists have often resorted to a so-called Wick rotation (in most cases, a heuristic argument at best) in order to apply Riemannian index theorems in a Lorentzian setting.
This can now be avoided; we will come back to an application in quantum field theory at the end of this section. 

To describe the setup let $M$ be a globally hyperbolic manifold whose Cauchy hypersurfaces are compact.
In other words, $M$ is spatially compact. 
Let $M$ carry a spin structure so that the spinor bundle $\Sigma M\to M$ and the classical Dirac operator acting on sections of $\Sigma M$ are defined.
Furthermore, we assume that $n=\dim(M)$ is even.
In this case the spinor bundle splits into subbundles of left-handed and right-handed spinors, $\Sigma M = \Sigma^LM \oplus \Sigma^RM$.
In $4$ dimensions, $\Sigma^LM=S_{1,0}M$ and  $\Sigma^RM=S_{0,1}M$.
The Dirac operator interchanges these two subbundles, i.e.\ with respect to the splitting  $\Sigma M = \Sigma^LM \oplus \Sigma^RM$ it takes the form 
$$
 {D} = \left( \begin{matrix} 0 &  D^- \\ D^+ & 0 \end{matrix} \right) \, .
$$
If $S\subset M$ is a smooth spacelike hypersurface then along $S$ we can write
$$
D = -\sigma\cdot\left(\nabla_\mathfrak{n} + iA -\frac{n-1}{2}H\right)
$$
where $\sigma$ is an invertible endomorphism field, $\mathfrak{n}$ is the past-directed unit normal field along $S$, $A$ is the Riemannian Dirac operator on $S$ and $H$ is the mean curvature of $S$.

Since $A$ is a self-adjoint elliptic differential operator on $S$ it has discrete real spectrum with all eigensections being smooth.
In particular, for any interval $I\subset \R$ we have the $L^2$-orthogonal projections $P_I(A)$ onto the sum of all eigenspaces of $A$ to the eigenvalues in $I$.

Now we assume that $M$ is a manifold with boundary, where the boundary is the disjoint union of two smooth spacelike Cauchy hypersurfaces $S_1$ and $S_2$.
Let $S_1$ lie in the past of $S_2$.
Denote the corresponding Riemannian Dirac operators by $A_1$ and $A_2$, respectively.
Now we can formulate the \emph{Atiyah-Patodi-Singer} boundary conditions.
A sufficiently regular left-handed spinor field $u$ on $M$ is said to satisfy the APS boundary conditions if $P_{[0,\infty)}(A_1)(u|_{S_1})=0$ and $P_{(-\infty,0]}(A_2)(u|_{S_2})=0$.

By $\FEAPS(M,\Sigma^LM)$ we denote the space of all left-handed spinor fields $u$ on $M$ which are continuous in time, $L^2$ in space, satisfy the APS boundary conditions, and are such that $Du$ is $L^2$.
This space of ``finite-energy sections'' naturally forms a Banach space.

In the same manner, we can define the space $\FEAPS(M,\Sigma^LM\otimes F)$ of twisted finite-energy spinors satisfying the APS conditions where $F$ is a Hermitian vector bundle equipped with a metric connection.

\begin{thm}[{B\"ar, Strohmaier \cite[Main thm.]{baer2015index}}]\label{thm:index}
Let $(M,g)$ be a compact globally hyperbolic Lorentzian manifold with boundary $\partial M = S_1 \sqcup S_2$. 
Here $S_1$ and $S_2$ are smooth spacelike Cauchy hypersurfaces, with $S_2$ lying in the future of $S_1$.
Assume that $M$ is even dimensional and comes equipped with a spin structure.
Let $F$ be a Hermitian vector bundle over $M$ equipped with a metric connection.

Then the twisted Dirac operator 
$$
D_\APS: \FEAPS(M;\Sigma^LM\otimes F) \to  L^2(M;\Sigma^RM\otimes F)
$$ 
under Atiyah-Patodi-Singer boundary conditions is Fredholm and its index is given by
\begin{align}
\mathrm{ind}[D_\APS]
=& 
\int_M \Adach (g)\wedge \mathrm{ch}(F) + \int_{\partial M}T(\Adach(g)\wedge \mathrm{ch}(F)) \notag\\
&- \frac{h(A_{1})+h(A_{2})+\eta(A_{1})-\eta(A_{2})}{2}\, .
\label{eq:indexformel}
\end{align}
\end{thm}

The right hand side in the index formula is formally exactly the same as in the original Riemannian Atiyah-Patodi-Singer index theorem.
Here $\Adach(g)$ is the $\Adach$-form manufactured from the curvature of the Levi-Civita connection of the Lorentzian manifold, $\mathrm{ch}(F)$ is the Chern character form of the curvature of $F$, $T(\Adach(g)\wedge\mathrm{ch}(F))$ is the corresponding transgression form which also depends on the second fundamental form of the boundary.
Moreover, $h$ denotes the dimension of the kernel, and $\eta$ the $\eta$-invariant of the corresponding operator.

The Dirac operator on a Lorentzian manifold is far from hypoelliptic; 
solutions of the Dirac equation $Du=0$ can in general have very low regularity.
Theorem~3.5 in \cite{baer2015index} tells us that under Atiyah-Patodi-Singer boundary conditions this is no longer so.
Solutions of $Du=0$ which satisfy the APS boundary condtions are always smooth as if we had elliptic regularity at our disposal.
Moreover,
$$
\mathrm{ind}[D_\APS] = \dim \ker\big[D_\APS|_{C^\infty(M;\Sigma^LM\otimes F)}\big] -  \dim \ker\big[D_\aAPS|_{C^\infty(M;\Sigma^LM\otimes F)}\big]\, .
$$ 
Here $D_\aAPS$ stands for the Dirac operator subject to \emph{anti-Atiyah-Patodi-Singer} boundary conditions, the conditions complementary to the APS conditions.
The occurrence of the aAPS boundary conditions and the fact that $D_\aAPS$ again maps sections of $\Sigma^LM\otimes F$ to those of $\Sigma^RM\otimes F$, and not in the reverse direction, are different from the corresponding formula in the Riemannian setting.
In the elliptic case, the Dirac operator with anti-APS boundary conditions will in general have infinite-dimensional kernel and thus not be Fredholm.
In contrast, in the Lorentzian situation, anti-APS boundary conditions work equally well as the APS conditions.

In the remainder of this section we sketch an application of Theorem~\ref{thm:index} in the context of algebraic quantum field theory (QFT) on curved spacetimes.
Details can be found in \cite{baer2015rigorous}.
The so-called \emph{$2$-point functions} are objects of central importance in such a QFT.
They are distributional bi-solutions of the Dirac equation on $M\times M$.
A particularly important class of $2$-point functions is determined by the \emph{Hadamard condition} which specifies the singular structure of these distributions.
Hence the difference $\omega_1-\omega_2$ of two Hadamard $2$-point functions $\omega_1$ and $\omega_2$ is a smooth bi-solution of the Dirac equation.
In this case we can associate a smooth $1$-form $J^{\omega_1,\omega_2}$, the \emph{relative current}.
The point here is that the definition of an absolute current $J^{\omega_i}$ would require a regularization procedure due to the singular nature of $\omega_i$ but the relative version is unambiguously defined and smooth.

A computation shows that $J^{\omega_1,\omega_2}$ is coclosed, $\delta J^{\omega_1,\omega_2}=0$.
Now if $S\subset M$ is a smooth spacelike Cauchy hypersurface with future-directed unit normal field $\mathfrak{n}$ then we can define the \emph{relative charge} by $Q^{\omega_1,\omega_2} = \int_S J^{\omega_1,\omega_2}(\mathfrak{n}) \, dS$.
Since $J^{\omega_1,\omega_2}$ is coclosed the divergence theorem implies that the definition of $Q^{\omega_1,\omega_2}$ is independent of the choice of $S$.

The Cauchy hypersurface $S$ in the above definition of the relative charge was just an auxiliary tool.
But using a Fock space construction we can also associate a $2$-point function $\omega_S$ to any smooth spacelike Cauchy hypersurface $S$.
This $\omega_S$ is to be thought of as the vacuum expectation value for an observer with spatial universe $S$.
In case the metric of $M$ and the connection of $F$ are of product form near $S$ it is known that $\omega_S$ is of Hadamard form. 
Thus we can define the relative charge $Q^{\omega_{S_1},\omega_{S_2}}$ for the two boundary parts if the metric of $M$ and the connection of $F$ are of product form near $\partial M$ which we now assume.

Working with left-handed spinors the main result in \cite{baer2015rigorous} relates $Q^{\omega_{S_1},\omega_{S_2}}$ to the index of the Dirac operator in Theorem~\ref{thm:index} and we obtain
\begin{equation}
Q_L^{\omega_{S_1},\omega_{S_2}}
:= 
Q^{\omega_{S_1},\omega_{S_2}}
= 
\int_M \Adach(g)\wedge\mathrm{ch}(F) 
-\frac{h(A_1) - h(A_2)+\eta(A_1)-\eta(A_2)}{2} \, .
\label{eq:QL}
\end{equation}
The product assumption on the metric of $M$ and connection of $F$ near the boundary implies that the transgression form vanishes near the boundary so that the boundary integral in \eqref{eq:indexformel} vanishes.
The opposite sign in front of $h(A_2)$ in equations \eqref{eq:indexformel} and \eqref{eq:QL} is not a misprint but due to the convention on how the eigenvalue $0$ is treated in the APS conditions.

Similarly, interchanging the roles of left-handed and right-handed spinors we obtain the ``right-handed'' relative charge
\begin{equation}
Q_R^{\omega_{S_1},\omega_{S_2}}
= 
-\int_M \Adach(g)\wedge\mathrm{ch}(F) 
+\frac{h(A_1) - h(A_2)+\eta(A_1)-\eta(A_2)}{2} \, .
\label{eq:QR}
\end{equation}
Hence the \emph{total relative charge} vanishes, $Q^{\omega_{S_1},\omega_{S_2}} := Q_R^{\omega_{S_1},\omega_{S_2}} + Q_L^{\omega_{S_1},\omega_{S_2}} = 0$, while the \emph{chiral relative charge} does not in general, 
\begin{align*}
Q_\mathrm{chir}^{\omega_{S_1},\omega_{S_2}} 
:=&
\,\, Q_R^{\omega_{S_1},\omega_{S_2}} - Q_L^{\omega_{S_1},\omega_{S_2}} \\
=&
-2\int_M \Adach(g)\wedge\mathrm{ch}(F) + h(A_1) - h(A_2)+\eta(A_1)-\eta(A_2) \, .
\end{align*}
Thus observers in $S_1$ and later in $S_2$ will disagree on the difference of numbers of left-handed and right-handed fermions.
Such quantities which are classically preserved but whose quantum counterparts are not are called \emph{anomalies}.
The \emph{chiral anomaly} treated here is a prominent example in the physics literature.
It explains the rate of decay of the neutral pion into two photons.
It is influenced by the gravitational field via $\Adach(g)$ and by an external field (e.g.\ electromagnetic) via $\mathrm{ch}(F)$.


\section{Green-hyperbolic operators}

In this section we describe a rather large class of linear differential operators of various orders introduced in \cite{MR3302643} which generalize normally hyperbolic and Dirac operators and are characterized by the existence of so-called Green's operators.
It turns out that these operators share many of the good solution properties of normally hyperbolic and Dirac operators.

Let $E_1,E_2\to M$ be vector bundles over a globally hyperbolic manifold $M$.
Let $P:C^\infty(M,E_1) \to C^\infty(M,E_2)$ be a linear differential operator.
There is a unique linear differential operator acting on the dual bundles, $P^*:C^\infty(M,E_2^*) \to C^\infty(M,E_1^*)$, characterized by
\begin{equation}
\int_M \phi(Pf) \dV = \int_M (P^*\phi)(f) \dV
\label{eq:FormalDual}
\end{equation}
for all $f\in C^\infty(M,E_1)$ and $\phi\in C^\infty(M,E_2^*)$ such that $\supp f \cap \supp(\phi)$ is compact.
The operator $P^*$ is called the \emph{formally dual operator} of $P$.

\begin{dfn}\label{def:GreenOp}
An \emph{advanced Green's operator} of $P$ is a linear map $G_+:C^\infty_c(M,E_2)\to C^\infty(M,E_1)$ such that
\begin{enumerate}[(i)]
\item\label{eq:GreenOp1}
$G_+Pf=f$ for all $f\in C^\infty_c(M,E_1)$;
\item\label{eq:GreenOp2}
$PG_+f=f$ for all $f\in C^\infty_c(M,E_2)$;
\item\label{eq:GreenOp3}
$\supp(G_+f)\subset J^+(\supp f)$ for all $f\in C^\infty_c(M,E_2)$.
\end{enumerate}
A linear map $G_-:C^\infty_c(M,E_2)\to C^\infty(M,E_1)$ is called a \emph{retarded Green's operator} of $P$ if \eqref{eq:GreenOp1}, \eqref{eq:GreenOp2}, and 
\begin{enumerate}[(i)']
\setcounter{enumi}{2}
\item \label{eq:GreenOp4}
$\supp(G_-f)\subset J^-(\supp f)$ holds for every $f\in C^\infty_c(M,E_2)$.
\end{enumerate}
\end{dfn}
 
\begin{dfn}
The operator $P$ is be called \emph{Green hyperbolic} if $P$ and $P^*$ have advanced and retarded Green's operators.
\end{dfn}

It turns out that the advanced and retarded Green's operators of a Green-hyperbolic operator are automatically unique.
Prime examples for Green-hyperbolic operators are normally hyperbolic operators \cite[Cor.~3.4.3]{MR2298021}, Dirac-type operators \cite[Cor.~3.15]{MR3302643}, and symmetric hyperbolic systems \cite[Thm.~5.9]{MR3302643}.

Another example is provided by the \emph{Proca operator} describing massive vector bosons.
Here $E_1=E_2=T^*M$ and $m$ is a positive constant.
Then the Proca operator is given by $P=\delta d +m^2$ where $d$ is the exterior differential and $\delta$ the codifferential.
It is of second order but not normally hyperbolic.
Yet it is Green-hyperbolic.

The class of Green-hyperbolic operators is closed under a number of natural operations: composition, taking direct sums, dualizing, and restriction to suitable subregions.

The advanced Green's operator can be extended to sections with past-compact support in such way that the conditions in Definition~\ref{def:GreenOp} remain valid.
By condition (iii) the resulting section will again have past-compact support, $G_+:C^\infty_{pc}(M,E_2)\to C^\infty_{pc}(M,E_1)$.
Hence on $C^\infty_{pc}(M,E_2)$ the operator $G_+$ is the inverse of $P$ itself restricted to sections with past-compact support.
In particular, $P$ is invertible as an operator $C^\infty_{pc}(M,E_1)\to C^\infty_{pc}(M,E_2)$.
Similarly, $G_-$ extends to an operator on sections with future-compact support, $G_-:C^\infty_{fc}(M,E_2)\to C^\infty_{fc}(M,E_1)$.

The difference $G:=G_+-G_-$, sometimes called the \emph{causal propagator}, maps $C^\infty_c(M,E_2)$ to $C^\infty_{sc}(M,E_1)$, the space of smooth sections with spatially compact support.
Important information about the solution theory of Green-hyperbolic operators is encoded in the following

\begin{thm}\label{thm:exactsequence}
The sequence
\begin{equation*}
 \{0\}\to C^\infty_{c}(M,E_1)\xrightarrow{P} C^\infty_{c}(M,E_2)\xrightarrow{G} C^\infty_{sc}(M,E_1)\xrightarrow{P} C^\infty_{sc}(M,E_2) \to  \{0\}
\end{equation*}
is exact.
\end{thm}

The operator $P$ itself and its advanced and retarded Green's operators extend to distributional sections.
These extensions have essentially the same properties;
in particular, the analogue of Theorem~\ref{thm:exactsequence} holds \cite[Thm.~4.3]{MR3302643}.
For applications in algebraic quantum field theory on curved spacetimes see e.g.\ \cite{MR3289848,MR2298021}.

%
%
%

\bibliographystyle{plain}

\end{document}